\documentclass[letterpaper, 10 pt, conference]{ieeeconf}  

\IEEEoverridecommandlockouts                              
\usepackage[T1]{fontenc}
\usepackage[utf8]{inputenc}
\usepackage{amsmath, amssymb, amsfonts}

\usepackage{amsthm}
\usepackage{mathtools,thmtools}
\usepackage{bm}
\usepackage{BOONDOX-ds} 
\let\labelindent\relax
\usepackage{enumitem}
\usepackage{booktabs}
\usepackage{graphicx}
\usepackage{tikz}
\usetikzlibrary{arrows.meta, calc, positioning}
\usepackage[colorlinks=true,
            linkcolor=black!70!blue,
            citecolor=black!55!green,
            urlcolor=black!45!cyan]{hyperref}
\usepackage{cleveref}
\usepackage[table, dvipsnames]{xcolor}
\definecolor{LightGray}{rgb}{0.93,0.93,0.93}
\usepackage{cite}
\usepackage{stmaryrd}
\usepackage[acronym]{glossaries} 
\usepackage{caption}
\newacronym{SRG}{SRG}{Scaled Relative Graph}
\newacronym{SIP}{SIP}{Semi-Inner Product}
\newacronym{JMT}{JMT}{James--{Mili\v ci\'c}--Tapia}
\newacronym{MDP}{MDP}{Markov Decision Process}

\declaretheorem[style=definition]{theorem}

\declaretheorem[style=definition]{proposition}
\declaretheorem[style=definition]{lemma}
\declaretheorem[style=definition,qed=$\vartriangle$]{remark}
\declaretheorem[style=definition,qed=$\blacktriangle$]{example}
\declaretheorem[style=definition,numbered=no]{standing assumption}

\declaretheorem[style=definition]{definition}

\newcommand{\R}{\mathbb{R}}

\newcommand{\C}{\mathbb{C}}

\newcommand{\Real}{\operatorname{{Re}}}
\newcommand{\Imag}{\operatorname{{Im}}}
\newcommand{\norm}[1]{\left\lVert #1 \right\rVert}

\newcommand{\sign}{\operatorname{sign}}
\DeclareMathOperator{\diag}{\textup{diag}} 

\newcommand{\graph}{\operatorname{gra}}
\newcommand{\igraph}{\mathrm{i}\!\operatorname{gra}}

\newcommand{\SRGL}{\mathrm{SRG}_L}
\newcommand{\SRGR}{\mathrm{SRG}_R}
\newcommand{\SRG}{\mathrm{SRG}}

\DeclareMathOperator{\lognorm}{\mu}
\DeclareMathOperator{\lab}{\textup{lab}}
\newcommand{\sip}[2]{\left[ #1,\, #2 \right]}
\newcommand{\wpair}[2]{\llbracket #1,\, #2 \rrbracket}
\newcommand{\Iinfty}[1]{\ensuremath{I_{\infty}(#1)}}
\newcommand{\cdotnorm}{\mspace{2mu}\cdot\mspace{2mu}}

\title{Scaled Relative Graphs in Normed Spaces}

\author{
A.~Padoan\thanks{A.~Padoan is with the Department of Electrical and Computer Engineering, University of British Columbia, Vancouver, BC V6T\,1Z4, Canada. \linebreak {E-mail}: {\tt\footnotesize alberto.padoan@ubc.ca}. The author acknowledges the support of the Natural Sciences and Engineering Research Council of Canada (NSERC). Grant numbers: RGPIN-2025-06895 and DGECR-2025-00382.
}}
\date{~}

\begin{document}
\maketitle

\begin{abstract}
The paper extends the \gls{SRG} framework
of Ryu, Hannah, and Yin from Hilbert spaces to normed spaces. Our
extension replaces the inner product with a regular pairing,
whose asymmetry gives rise to directional angles and, in turn,
directional \glspl{SRG}. Directional \glspl{SRG} are shown to provide geometric containment tests certifying key operator properties, including contraction and monotonicity. Calculus rules for \glspl{SRG} under scaling, inversion, addition, and composition are also derived. The theory is illustrated by numerical examples, including a graphical analysis of the contractivity of Bellman operators using directional \glspl{SRG}.
\end{abstract}

\glsreset{SRG}
\glsreset{SIP}

\section{Introduction}
\label{sec:intro}

Graphical methods have a long and distinguished history in the analysis of feedback systems~\cite{nyquist1932,bode1945,popov1961,zames1966,megretski1997}, offering intuition that purely analytic treatments often obscure. 
The \gls{SRG}, recently proposed by Ryu, Hannah, and Yin~\cite{ryu2022scaled}, is a powerful heir to this tradition. 
The \gls{SRG} associates 
to an operator $T$ on a Hilbert space, with graph $\graph(T)$,  a subset of the extended complex plane, defined as
\begin{equation} \label{eq:srg_hilbert}
\!\!\SRG(T) \!=\! \left\{ \frac{\norm{y}}{\norm{x}} e^{\pm i \angle(x,y)}
\middle|\, (x,y) \in \graph(T) - \graph(T) \right\}\! , \!\!
\end{equation}
where $\norm{\cdotnorm}$ and $\angle(\cdot\,,\cdot)$ are the norm and angle induced by the inner product, respectively. The \gls{SRG} captures key operator properties, such as contraction and monotonicity, in a form amenable to geometric reasoning~\cite{ryu2022scaled}. The result is a unified framework for the convergence analysis of a broad class of algorithms~\cite{facchinei2003,beck2017,ryu2022book,bauschke2017}, including first-order methods and their proximal variants~\cite{beck2017,ryu2022book,bauschke2017}, operator splitting schemes~\cite{lions1979,ryu2022book,bauschke2017}, and monotone inclusion problems~\cite{facchinei2003,ryu2022book,bauschke2017}, as well as compositional stability tests reminiscent of the Nyquist criterion for general nonlinear 
operators~\cite{chaffey2023graphical}.

A fundamental limitation of the existing theory is that the \gls{SRG} has been defined only for operators on Hilbert spaces~\cite{chaffey2023graphical,pates2021,ryu2022scaled,chaffey2022rolledoff,chaffey2022circuit,krebbekx2025lure,vandeneijnden2025phase,baronprada2025mixed,baronprada2025mimo, nauta2026computing}, whose privileged role stems from the inner-product structure underlying the analysis of both optimization algorithms~\cite{ryu2022book,bauschke2017} and feedback systems~\cite{chaffey2023graphical}.
In many applications, however, signals are naturally sparse or
bounded, making $\ell^1$ and~$\ell^\infty$ more natural norms
than~$\ell^2$. Neither is induced by an inner product, so the entire \gls{SRG} toolkit --- convergence certificates, stability tests, and compositional calculus --- is unavailable for operators on these spaces. Examples include
$\ell^1$ optimal
controllers~\cite{dahleh1994control}, Bellman
operators in $\ell^\infty$ from dynamic programming~\cite{bertsekas2012dynamic,sutton2018reinforcement}, and fixed-point iterations in Banach spaces~\cite{Zeidler1986,Chidume2009}.

The paper extends the \gls{SRG} to normed spaces by replacing the
inner product with a \emph{regular pairing}~\cite{DavydovJafarpourBullo2022,proskurnikov2024regular,bullo2024,davydov2024monotone}, which exists in every normed space, but need not be 
symmetric. The resulting \gls{SRG} framework offers graphical characterizations of operator properties and a compositional calculus for incremental stability analysis in non-Euclidean geometries. 

\emph{Contributions.}
First, using regular pairings as substitutes for the inner product, we introduce \emph{directional cosines} and \emph{angles} in normed spaces and show that, in $\ell^1$ and $\ell^\infty$, they capture the polyhedral geometry of the unit balls through sign-pattern and facet transitions (Section~\ref{sec:geometry}).
Second, we use directional angles to derive a gain--phase decomposition of the logarithmic norm and a phase characterization of monotonicity (Section~\ref{sec:geometry}).
Third, we define \emph{directional \glspl{SRG}} in normed spaces,
characterize operator properties via geometric containment tests, and establish a calculus of \gls{SRG} operations extending that of~\cite{ryu2022scaled,chaffey2023graphical} (Section~\ref{sec:srg}).  Fourth, numerical examples illustrate how directional \glspl{SRG}
capture the geometry of monotonicity for linear and nonlinear
operators and suggest that standard Lipschitz estimates for Bellman
operators in $\ell^\infty$ may be conservative (Section~\ref{sec:numerical}).

\emph{Paper organization.}
Section~\ref{sec:sip} reviews weak and regular pairings from~\cite{DavydovJafarpourBullo2022,proskurnikov2024regular,bullo2024}. 
Section~\ref{sec:geometry} develops the geometry and operator theory of directional angles induced by regular pairings.
Section~\ref{sec:srg} defines directional
\glspl{SRG}, relates operator properties to geometric containment
conditions, and establishes the \gls{SRG} calculus.
Section~\ref{sec:numerical} presents our numerical case studies.
All proofs are collected in the Appendix.

\emph{Notation.}
Our conventions largely
follow~\cite{ryu2022book,chaffey2023graphical,proskurnikov2024regular}.
Standard basis vectors of $\R^n$ are $e_1,\dots,e_n$ and
the transpose of ${A\in\R^{p\times m}}$ is  $A^\top$. 
The imaginary unit is ${\textrm{i}\in\C}$. The real and imaginary parts, complex conjugate, modulus, and argument of ${z\in\C}$ are $\Real z$, $\Imag z$, $\bar z$, $|z|$,
and $\angle z$, respectively. 
The {extended complex 
plane} is ${\bar\C=\C\cup\{\infty\}}$, with arithmetic rules as
in~\cite{ryu2022scaled,chaffey2023graphical}.
An operator $T$ on a set $X$ is a possibly set-valued map, identified with its graph ${\graph(T) = \{(x,y) \mid y \in T(x)\}}$. We denote by~$I$ the identity operator and define
${\igraph(T)=\graph(T)-\graph(T)}$.
Scaling, inversion, addition, and composition of operators are defined
as in~\cite{ryu2022book,chaffey2023graphical}. 
In a normed space ${(X,\|\cdot\|)}$, the unit sphere is ${S_X=\{x\in X \mid \|x\|=1\}}$, the norm of a bounded linear operator $A$ is the induced norm ${\|A\|=\sup_{x\in S_X}\|Ax\|}$, and
the {logarithmic norm} of $A$
is $\mu(A)=\lim_{h\to 0^+}(\|I+hA\|-1)/h$~\cite{bullo2024}. 
For $p\in[1,\infty]$, $\|\cdot\|_p$ is the $\ell^p$ norm
on~$\R^n$ and $\mu_p(A)$ is the induced logarithmic norm of $A\in\R^{n\times n}$. For $x\in\R^n$, ${\Iinfty{x}=\{i\mid |x_i|=\|x\|_\infty\}}$
is the set of {peak indices} and $m_x=\min\Iinfty{x}$ is the
{minimal peak index}~\cite{proskurnikov2024regular}. 
The {map} ${\sign\colon\R^n\to\R^n}$ acts componentwise:
$(\sign(x))_i = x_i/|x_i|$ if $x_i\neq 0$ and  $(\sign(x))_i = 0$ otherwise.

\section{Weak and Regular Pairings}
\label{sec:sip}

The Jordan--von Neumann theorem~\cite{jordan1935} establishes that a real normed space $(X,\norm{\cdotnorm})$ is an inner product space if and only if its norm satisfies the \textit{parallelogram identity}
\begin{equation}
  \label{eq:parallelogram}
  \|x+y\|^2 + \|x-y\|^2 = 2\|x\|^2 + 2\|y\|^2,
  \quad \forall\, x,y \in X.
\end{equation}
The identity~\eqref{eq:parallelogram} fails, \textit{e.g.}, in $\ell^1(\R^n)$ and $\ell^\infty(\R^n)$ for ${n \ge 2}$, so neither space admits an inner product. Lumer~\cite{lumer1961semi} and 
Giles~\cite{giles1967classes} introduced \glspl{SIP} as natural 
substitutes for inner products in normed spaces. More general notions have since been proposed~\cite{DavydovJafarpourBullo2022,davydov2024monotone,bullo2024,proskurnikov2024regular}; in this work, we leverage the hierarchy of
\textit{regular pairings} studied in~\cite{proskurnikov2024regular}.

A \emph{weak pairing} is a binary operation ${\wpair{\cdot}{\cdot}: X \times X \to \R}$ satisfying subadditivity in the first argument, weak homogeneity, positive definiteness, and the Cauchy--Schwarz inequality~\cite{DavydovJafarpourBullo2022,davydov2024monotone,bullo2024,proskurnikov2024regular}. A weak pairing $\wpair{\cdot}{\cdot}$ is \emph{compatible} with the norm if ${\wpair{x}{x}=\|x\|^2}$ for all ${x\in X}$. \emph{Regular pairings} are a distinguished class of weak pairings, characterized by several equivalent conditions~\cite[Thm.~3.5]{proskurnikov2024regular}, including
\begin{enumerate}[label=\textup{(R\arabic*)}, noitemsep, leftmargin=*]
\item\label{rwp:lumer}
        \textit{Lumer's identity} for any bounded linear operator $A$,
        \begin{equation} \label{eq:lumer}
        \displaystyle \mu(A)=\sup_{x\in S_X}\wpair{Ax}{x} ,            
        \end{equation}
  \item\label{rwp:straight}
        the \textit{straight angle property},
        \begin{equation} \label{eq:straight_angle}
            \wpair{-x}{x}=-\norm{x}^2 \quad \forall x\in X ,
        \end{equation}
   \item\label{rwp:linear}
           and \textit{partial linearity},
          \begin{equation} \label{eq:partial_linearity}
              \wpair{x+ay}{y}=\wpair{x}{y}+a\norm{y}^2 \ \, \forall \, x,y\in X, \, a\in\R .
          \end{equation}
\end{enumerate}

\noindent
Every \gls{SIP} in the sense of 
Lumer~\cite{lumer1961semi} and Giles~\cite{giles1967classes} is a 
regular pairing that is linear (not merely subadditive) in its first 
argument~\cite{proskurnikov2024regular}.  It is well known that every normed space admits at least one \gls{SIP} compatible with the norm~\cite{lumer1961semi,giles1967classes,proskurnikov2024regular}. Another important class 
of regular pairings arises directly from the norm: the \textit{upper} and \textit{lower 
\gls{JMT} pairings}~\cite{james1947orthogonality,milicic1971semiproduct,
proskurnikov2024regular}, defined as
\begin{equation}
\label{eq:jmt}
  \sip{x}{y}_{\pm}
  \,=\,
  \norm{y}\lim_{t\to 0^{\pm}}\frac{\norm{y+tx}-\norm{y}}{t},
  \quad \forall\,x,y\in X.
\end{equation}
The \gls{JMT} pairings capture a key sensitivity property: 
$\sip{x}{y}_\pm$ are one-sided directional derivatives of $\norm{y}$ along the 
direction~$x$, scaled by $\norm{y}$.  If the norm is G\^ateaux differentiable, the upper and lower \gls{JMT} pairings coincide; in this case, the norm is induced by a \textit{unique} regular pairing, \textit{i.e.}, the only \gls{SIP} compatible with the norm~\cite{proskurnikov2024regular}. For polyhedral norms, such as 
$\ell^1$ and $\ell^\infty$, G\^ateaux differentiability fails and 
multiple regular pairings coexist~\cite{proskurnikov2024regular}.
Table~\ref{tab:pairings} summarizes selected regular pairings for $\ell^p(\R^n)$ from~\cite{proskurnikov2024regular}.

\begin{table}[h!]
\centering
\rowcolors{2}{LightGray}{white}
\small
\resizebox{\linewidth}{!}{%
\renewcommand{\arraystretch}{1.} %
\begin{tabular}{%
p{0.07\linewidth}%
p{0.51\linewidth}%
p{0.32\linewidth}%
}
\toprule
Norm & Regular pairing & Logarithmic norm \\
\midrule
$\begin{aligned}
\|x\|_2
\end{aligned}$
&
$\begin{aligned}
\wpair{x}{y}_2 &= y^\top x
\end{aligned}$
&
$\begin{aligned}
\tfrac{1}{2}\lambda_{\max}(A+A^\top)
\end{aligned}$
\\[.5ex]

$\begin{aligned}
\|x\|_1
\end{aligned}$
&
$\begin{aligned}
\wpair{x}{y}_1 &= \|y\|_1\sign(y)^\top x
\end{aligned}$
&
$\begin{aligned}
\max_j \Big(a_{jj}+\sum_{i\neq j}|a_{ij}|\Big)
\end{aligned}$
\\[.5ex]

$\begin{aligned}
\|x\|_\infty
\end{aligned}$
&
$\begin{aligned}
\wpair{x}{y}_\infty &= \max_{i\in\Iinfty{y}} x_i y_i
\end{aligned}$
&
$\begin{aligned}
\max_i \Big(a_{ii}+\sum_{j\neq i}|a_{ij}|\Big)
\end{aligned}$
\\[.5ex]

$\begin{aligned}
\|x\|_\infty
\end{aligned}$
&
$\begin{aligned}
\wpair{x}{y}_{\infty,m}
&= \|y\|_\infty \sign(y_{m_y})x_{m_y}
\end{aligned}$
&
$\begin{aligned}
\max_i \Big(a_{ii}+\sum_{j\neq i}|a_{ij}|\Big)
\end{aligned}$
\\
\bottomrule
\end{tabular}%
}
\caption{Norms, regular pairings, and logarithmic norms in $\ell^p(\R^n)$~\cite{proskurnikov2024regular}.}
\label{tab:pairings}
\end{table}

For $\ell^2(\R^n)$, the unique regular pairing $\wpair{x}{y}_{2}$
is the Euclidean inner product. For $\ell^1(\R^n)$, the \emph{sign
pairing} $\wpair{x}{y}_1$ is both a \gls{SIP} and a \gls{JMT}
pairing. For $\ell^\infty(\R^n)$, several choices are possible. The
\emph{max pairing} $\wpair{x}{y}_{\infty}$ is the upper \gls{JMT}
pairing and is permutation-invariant, but not a \gls{SIP}. The
\emph{min-index pairing} $\wpair{x}{y}_{\infty,m}$ is a \gls{SIP},
but neither a \gls{JMT} pairing nor permutation-invariant. Since
all regular pairings satisfy Lumer's
identity~\eqref{eq:lumer}, the choice does not affect the
logarithmic norm; it does, however, affect angles and \glspl{SRG}.

\begin{example}[{Sign, max, and min-index  pairings in $\R^2$}]
\label{ex:sip_R2}
Let ${x=(1, 0.5)}$ and ${y=(0.3, 1)}$. Then
${\norm{x}_1=1.5}$, ${\norm{y}_1=1.3}$, ${\norm{x}_\infty=1}$, ${\norm{y}_\infty=1}$,
${\Iinfty{x}=\{1\}}$, and  ${\Iinfty{y}=\{2\}}$.  Evaluating sign, max, and min-index pairings yields
\begin{equation}\nonumber
\begin{array}{rcl}
  \wpair{x}{y}_1 = \!&\!1.95\!&\! = \wpair{y}{x}_1,\\
  \wpair{x}{y}_\infty =  0.5 \!\!&\!\!\ne\!\!&\!\! 0.3 =  \wpair{y}{x}_\infty,\\
  \wpair{x}{y}_{\infty,m} = 0.5 \!\!&\!\!\ne\!\!&\!\! 0.3 = \wpair{y}{x}_{\infty,m}.
\end{array}
\end{equation}
The sign pairing value is symmetric, as $\sign(x)=\sign(y)$; this is
a coincidence, not a general property. Both $\ell^\infty$ regular pairing values are
asymmetric, reflecting that $x$ and $y$ attain their peak values at different components.
\end{example}

All our results depend on the choice of norm and regular
pairing, which we fix throughout unless otherwise stated.

\begin{standing assumption}
\label{sa:pairing}
The pair $(X,\|\cdot\|)$ is a real normed space equipped with a compatible regular pairing $\wpair{\cdot}{\cdot}$.
\end{standing assumption}

\section{Directional Angles and Cosines}
\label{sec:geometry}

In a real inner product space $(X,\langle\cdot,\cdot\rangle)$, the angle
\begin{equation}
\label{eq:angle_inner_product}
\angle(x,y) = \arccos\!\left(\frac{\langle x,y\rangle}{\norm{x}\norm{y}}\right)
\end{equation}
is symmetric, i.e., $\angle(x,y)=\angle(y,x)$. By contrast, in a normed space $(X,\|\cdot\|)$ a compatible regular pairing need not be symmetric, giving rise to \emph{directional} angles.

\begin{definition}[Left cosines and angles] \label{def:left_cosines_and_angles}
The \emph{left cosine} between nonzero vectors $x,y \in X$ is 
\begin{equation}
  \label{eq:cosines}
  \cos_L(x,y)
  \,=\, \frac{\wpair{y}{x}}{\norm{x}\norm{y}},
\end{equation}
and the corresponding \emph{left angle} between $x$ and $y$ is
\begin{equation}
\label{eq:angles}
  \angle_L(x,y) \,=\, \arccos\bigl(\cos_L(x,y)\bigr) .
\end{equation}
\end{definition}

\begin{figure*}[t!]
\centering

\begin{tikzpicture}[scale=1.15, >=Stealth]
  \begin{scope}[xshift=-3.5cm]
    \filldraw[fill=gray!10!white, draw=gray!50!white, thick]
      (1,0)--(0,1)--(-1,0)--(0,-1)--cycle;
    \draw[->](-1.35,0)--(1.45,0) node[right]{\small$x_1$};
    \draw[->](0,-1.35)--(0,1.45) node[above]{\small$x_2$};
    \draw[->,thick,red](0,0)--(0.65,0.35)
          node[yshift=-0.3cm,font=\small]{$x$};
    \draw[->,dashdotted,thick,blue!60!black](0,0)--(0.2,0.8)
          node[yshift=-0.4cm,xshift=0.1cm,font=\small]{$y$};
    \draw[->,thick,white!70!black](0.2,0.8)--++(0.22,0.22)
          node[yshift=0.1cm,xshift=0.6cm,font=\small]{$\sign(y)$};
    \draw[->,thick,red!70!black](0.65,0.35)--++(0.22,0.22)
          node[yshift=0.0cm,xshift=0.6cm,font=\small]{$\sign(x)$};
    \draw[red!70!black,ultra thick](1,0)--(0,1);
    \draw[blue!60!black, dashed, ultra thick](1,0)--(0,1);
    \draw[-stealth,orange!80!black,very thick,bend right=25,xshift=0.1cm,yshift=-0.2cm]
          (0.55,0.72)to(0.22,1.02);
    \node[orange!80!black,font=\small,xshift=.1cm,yshift=-.25cm] at (0.62,1.0){$\mathsf{S}$};
    \node[font=\small,align=center] at (0,-1.7)
      {$\ell^1$: sign-pattern transition \\
       \textcolor{orange!80!black}{$\mathsf{S}\colon(+,+)\mapsto(+,+)$}};
  \end{scope}

  \begin{scope}[xshift=1.5cm]
    \filldraw[fill=gray!10!white, draw=gray!50!white,thick](0,0) circle(1);
    \draw[->](-1.35,0)--(1.45,0) node[right]{\small$x_1$};
    \draw[->](0,-1.35)--(0,1.45) node[above]{\small$x_2$};
    \draw[->,thick,red](0,0)--({cos(28)},{sin(28)})
          node[xshift=-.1cm,yshift=-.4cm,font=\small]{$x$};
    \draw[->,dashdotted,thick,blue!60!black](0,0)--({cos(78)},{sin(78)})
          node[xshift=.2cm,yshift=-.25cm, font=\small]{$y$};
    \draw[->,thick,white!70!black]({cos(78)},{sin(78)})--
          ++({0.28*cos(78)},{0.28*sin(78)})
          node[xshift=.6cm,yshift=.1cm,font=\small]{${y/\!\norm{y}_2}$};
    \draw[->,thick,red!70!black]({cos(28)},{sin(28)})--
          ++({0.28*cos(28)},{0.28*sin(28)})
          node[xshift=.6cm,font=\small]{${x/\!\norm{x}_2}$};
    \draw[-stealth,orange!70!black,very thick]
          ({0.45*cos(28)},{0.45*sin(28)}) arc(28:78:0.45);
    \node[orange!70!black,font=\small,xshift=0.35cm,yshift=-0.125cm] at (0.12,0.60){$\mathsf{R}$};
    \node[font=\small,align=center] at (0,-1.7)
      {$\ell^2$: continuous rotation \\
       \textcolor{orange!80!black}{$\mathsf{R}\colon x\mapsto
       \bigl[\begin{smallmatrix}
       \cos\phi & {-}\sin\phi \\
       \sin\phi & ~~\cos\phi\end{smallmatrix}\bigr]x$}};
  \end{scope}

  \begin{scope}[xshift=7cm]
    \filldraw[fill=gray!10!white, draw=gray!50!white, thick]
      (-1,-1)--(1,-1)--(1,1)--(-1,1)--cycle;
    \draw[->](-1.35,0)--(1.45,0) node[right]{\small$x_1$};
    \draw[->](0,-1.35)--(0,1.45) node[above]{\small$x_2$};
    \draw[->,thick,red](0,0)--(1,0.42)
          node[xshift=-0.2cm,yshift=-0.35cm,font=\small]{$x$};
    \draw[->,dashdotted,thick,blue!60!black](0,0)--(0.28,1)
          node[xshift=-0.425cm,yshift=-0.5cm,font=\small]{$y$};
    \draw[->,thick,white!70!black](0.28,1)--++(0,0.28)
          node[xshift=0.9cm,yshift=0.2cm,font=\small]{$\sign(y_{m_y})e_{m_y}$};
    \draw[->,thick,red!70!black](1,0.42)--++(0.28,0)
          node[xshift=0.8cm,yshift=0.3cm,font=\small]{$\sign(x_{m_x})e_{m_x}$};
    \draw[red!70!black,ultra thick](1,-1)--(1,1);
    \node[red!70!black,below,font=\small] at (1.22,0){$F_1^+$};
    \draw[blue!60!black,ultra thick](-1,1)--(1,1);
    \node[blue!60!black,left,font=\small] at (0,1.22){$F_2^+$};
    \draw[-stealth,very thick,orange!80!black,bend left=40]
          (1.,0.5)to(0.35,1.);
    \node[orange!80!black,xshift=-.55cm,yshift=-0.05cm,font=\small] at (1.18,0.85){$\mathsf{F}$};
    \draw[red!50!black,dashed,thick](0.28,1)--(0.28,0);
    \filldraw[red!50!black](0.28,0) circle(1.4pt)
          node[xshift=.1cm,yshift=-.35cm,font=\small]{$y_{m_x}$};
    \node[font=\small,align=center] at (0,-1.7)
      {$\ell^\infty$: facet transition \\
       \textcolor{orange!80!black}{$\mathsf{F}\colon F_1^+\mapsto F_2^+$}};
  \end{scope}

\end{tikzpicture}
\caption{%
Directional angles in $\ell^1$ (left), $\ell^2$
(center), and $\ell^\infty$ (right) between unit-norm vectors 
\textcolor{red}{$x$} (solid) and \textcolor{blue}{$y$} (dashdotted) in $\R^2$.  
The outward normals at \textcolor{red}{$x$} are 
\textcolor{red!70!black}{$\sign(x)$}  in $\ell^1$, 
\textcolor{red!70!black}{$x/\norm{x}_2$}  in $\ell^2$, and 
\textcolor{red!70!black}{$\sign(x_{m_x})e_{m_x}$} in $\ell^\infty$ with the min-index pairing.  
In $\ell^1$, an angle encodes transitions between sign patterns on the unit cross-polytope;
in $\ell^2$, it encodes continuous rotations on the unit sphere;
in $\ell^\infty$, it encodes transitions between active facets of the unit hypercube.
}
\label{fig:phase_transitions}
\end{figure*}
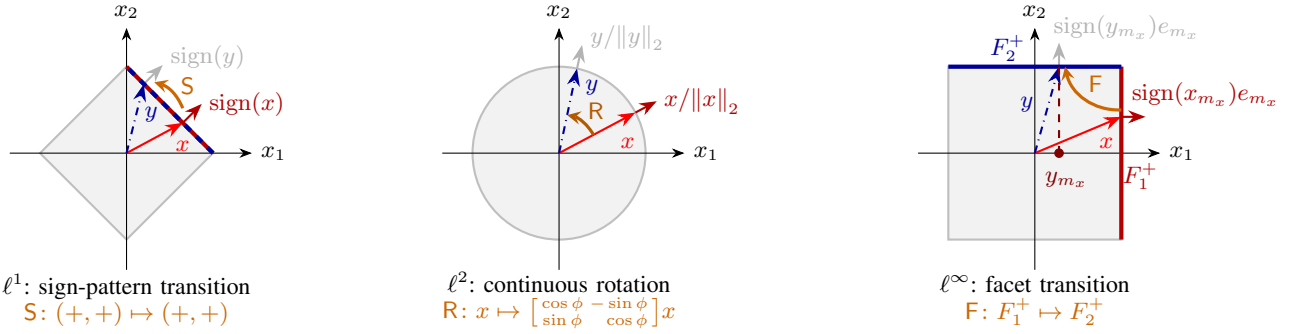

\noindent
By the Cauchy--Schwarz inequality, ${\cos_L(x,y)\in[-1,1]}$ for all nonzero ${x,y\in X}$, so the left angle is well defined. The \emph{right} cosine and angle are defined by swapping arguments, namely  ${\cos_R(x,y)=\cos_L(y,x)}$ and ${\angle_R(x,y)=\angle_L(y,x)}$ for all nonzero ${x,y\in X}$. Angles and cosines are called \emph{directional} if no distinction between left and right is needed. 
When $\wpair{\cdot}{\cdot}$ is a \gls{SIP}, linearity in the first
argument gives ${\cos_L(x,-y) = -\cos_L(x,y)}$ for all nonzero
${x, y \in X}$. When  $\wpair{\cdot}{\cdot}$ is an inner product, directional angles coincide, since it is both symmetric and the unique regular pairing
compatible with the norm~\cite{proskurnikov2024regular}.
The following result summarizes useful properties of left cosines; the corresponding results for the right cosines follow by swapping arguments.

\begin{lemma}[Directional cosine identities]
\label{lem:cosine_properties}
Let $x, y, z \in X$ be nonzero, let $\sigma > 0$, and let
${\alpha, \beta \geq 0}$. Then
\begin{enumerate}[label=\textup{(\alph*)}, noitemsep, leftmargin=*]
  \item\label{cos:self}
        $\cos_L(x,\pm x)= \pm 1$;
  \item\label{cos:scale}  ${\cos_L(\sigma x,y)=\cos_L(x,\sigma  y)=\cos_L(x,y)};$   
  \item\label{cos:convex} for ${w = \alpha y + \beta z}$, with ${w \neq 0}$, there holds
$$
\cos_L(x, w)
\le
\alpha\tfrac{\norm{y}}{\norm{w}}\cos_L(x,y)
+\beta\tfrac{\norm{z}}{\norm{w}}\cos_L(x,z);
$$
\item\label{cos:defect} 
if $\wpair{\cdot}{\cdot}$ is a \gls{SIP} and ${x, y, z \in S_X}$, then
        $$\resizebox{\linewidth}{!}{$
          \bigl|\cos_L(x,z) - \cos_L(y,z)\,\cos_L(x,y)\bigr|
          \leq \norm{z - \cos_L(y,z)\,y}.
$
}$$
\end{enumerate}
\end{lemma}
 
\begin{example}[Directional cosines in $\R^2$]
\label{ex:cosines_R2}
Let $x=(1,0.5)$ and $y=(0.3,1)$,
as in Example~\ref{ex:sip_R2}. Then $\sign(x)=(1,1) $ and $\sign(y)=(1,1)$. Moreover, $m_x=1$ and $m_y=2$.  Then
$$
\begin{array}{rrcl}
\ell^1: \ &
\cos_L(x,y) = \!&\!1\!&\! = \cos_R(x,y),\\
\ell^2: \ &
\cos_L(x,y)= \!&\!\tfrac{16}{\sqrt{545}}\!&\! = \cos_R(x,y),\\
\ell^\infty: \ &
\cos_L(x,y)=\tfrac{3}{10} \!\!&\!\!\neq\!\!&\!\! \tfrac{1}{2} = \cos_R(x,y).
\end{array}
$$
In $\ell^1$, both cosines equal one since $x$ and $y$ share the
same sign pattern; in $\ell^2$, they coincide by symmetry of the
inner product. In $\ell^\infty$, they differ because $x$ and $y$
attain their maxima at different indices. The max and min-index
pairings agree here, as $x$ and $y$ have a single peak index.
\end{example}

\subsection{Directional cosines and polyhedral norms}

In inner product spaces, angles correspond to \emph{continuous} rotations.
In normed spaces with polyhedral norms, directional angles instead encode \emph{combinatorial} transitions between facets of the unit ball, whose structure depends on the regular pairing.  We illustrate this contrast in $\R^n$ by comparing the rotational geometry of $\ell^2$ with the dual polyhedral geometries of
$\ell^1$ and $\ell^\infty$.

\subsubsection{Directional cosines in \texorpdfstring{$\ell^1$}{l1}}
The unit ball of $\ell^1(\R^n)$ is the cross-polytope whose $2^n$ open facets
are indexed by sign vectors $\sigma\in\{-1,+1\}^n$, defined as
\begin{equation}
  G_\sigma = \bigl\{x\in\R^n \,\mid\, \norm{x}_1=1,\;\sign(x)=\sigma\bigr\}.
\end{equation}
The outward normal to $G_\sigma$ is $\sigma$ itself. 
For nonzero ${x,y\in \R^n}$, the left cosine 
${\cos_L(x,y) = \sign(x)^\top y /\norm{y}_1}$
induced by the sign pairing is the Euclidean inner product between the outward normal $\sign(x)$ and
the normalized vector $y/\norm{y}_1$.
Introducing the sign-concordant and sign-discordant index sets
$I_\pm(x,y) = \bigl\{i\in\{1,\ldots,n\} \,\mid\, \pm \, x_iy_i > 0\bigr\},$
the left cosine can be written as
\begin{equation}
  \label{eq:cosL_l1_mass}
  \cos_L(x,y)
  = \sum_{i\in I_+(x,y)}\tfrac{|y_i|~}{\norm{y}_1} -\sum_{i\in I_-(x,y)}\tfrac{|y_i|~}{\norm{y}_1} ,
\end{equation}
which reflects how the $\ell^1$ ``mass'' of $y$ splits between sign-concordant and sign-discordant indices relative to $x$. From this viewpoint, the map 
${\mathsf{S}\colon\sign(x)\mapsto\sign(y)}$, which we call the 
\emph{sign-pattern transition}, plays the role in $\ell^1$ that a 
rotation plays in $\ell^2$. The map $\mathsf{S}$ identifies 
\emph{which} facet transition occurs, while the left cosine 
$\cos_L(x,y)$ quantifies \emph{how much} of the $\ell^1$ mass 
of~$y$ is redistributed between sign-concordant and sign-discordant 
indices; see Figure~\ref{fig:phase_transitions} (left).
 
\subsubsection{Directional cosines in \texorpdfstring{$\ell^\infty$}{l-infinity}}
The unit ball of $\ell^\infty(\R^n)$ is the hypercube $[-1,1]^n$, with $2n$
closed facets defined as
\begin{equation}
  F_k^\pm = \bigl\{x\in[-1,1]^n \,\mid\, \pm x_k = 1\bigr\},
  \ \,  k\in\{1,\ldots,n\}.
\end{equation}
The outward unit normal to $F_k^\pm$ is $\pm e_k$. 
A facet $F_k^\pm$ is \emph{active} at ${x\in\R^n}$ if ${k\in\Iinfty{x}}$.
The choice of regular pairing directly affects the geometry of directional cosines, as we now illustrate for the max and min-index pairings.

\paragraph{Min-index pairing} The min-index pairing induces the
left cosine ${\cos_L(x,y)=\sign(x_{m_x})\,e_{m_x}^{\top} y/\norm{y}_\infty}$, \textit{i.e.}, the signed, $\norm{y}_\infty$-normalized value of~$y$ at
$x$'s minimal peak index. All other components of~$y$ are ``invisible''
to the cosine. The \emph{facet label map} ${\lab_{m}\colon x\mapsto(m_x,\,\sign(x_{m_x}))}$ assigns to each nonzero ${x\in\R^n}$ the active facet of the unit hypercube at~$x$. The \emph{facet transition} ${\mathsf{F}\colon\lab_{m}(x)\mapsto\lab_{m}(y)}$ then captures a discrete ``jump'' among the $2n$ facets. The map
$\mathsf{F}$ is the $\ell^\infty$ analogue of a rotation in $\ell^2$: it identifies \emph{which} facet jump occurs, while $\cos_L(x,y)$ quantifies \emph{how much}~$y$ aligns with~$x$ at $x$'s peak index; see Figure~\ref{fig:phase_transitions} (right).

\paragraph{Max pairing} The max pairing induces the left cosine
${\cos_L(x,y)=\max_{i\in\Iinfty{x}}\sign(x_i) y_i/\norm{y}_\infty}$,
\textit{i.e.}, the largest signed, $\norm{y}_\infty$-normalized
alignment between $x$ and $y$ over all peak indices of~$x$. The max and min-index
pairings coincide when $|\Iinfty{x}|=1$; when multiple peak indices
exist, the max pairing selects the most favorable, yielding the
smallest left angle among \textit{all} regular pairings
for~$\ell^\infty$. The \emph{facet label map}
${\lab_\infty\colon x\mapsto(\Iinfty{x},\,\sign(x)|_{\Iinfty{x}})}$
assigns to each nonzero ${x\in\R^n}$ the \emph{set} of active
facets of the unit hypercube at~$x$. The \emph{facet transition}
${\mathsf{F}\colon\lab_\infty(x)\mapsto\lab_\infty(y)}$ then
captures a set-valued ``jump'' among facets, rather than the single
jump tracked by the min-index pairing.

\begin{remark}[Duality] \label{rem:duality}
The Banach space duality between $\ell^1$ and
$\ell^\infty$ in $\R^n$ manifests itself geometrically in different ways.
In $\ell^1$, the cosine depends on \emph{all} components of~$y$
(a \textit{global} quantity), while in $\ell^\infty$ it depends on a
\emph{single} component (a \textit{local} quantity). Duality also
appears in worst-case inputs: in $\ell^1$ inputs are maximally sparse (unit
impulses), while in $\ell^\infty$ inputs are maximally spread
(``bang-bang'' signals). For linear operators, duality yields the logarithmic norm
identity ${\mu_1(A) = \mu_\infty(A^\top)}$~\cite{bullo2024}.
\end{remark}

\subsection{Directional phase of an operator}

Directional angles naturally give rise to corresponding notions of phase for an operator.
The \emph{left phase} of an operator~$T$ at ${(x,y)\in\graph(T)}$, with ${x,y\neq 0}$,
is defined by $\angle_L(x,y)$; similarly, the \emph{incremental left phase} 
of~$T$ at ${(x,y)\in\igraph(T)}$, with ${x,y\neq 0}$, is likewise given by $\angle_L(x,y)$. Right phases are defined analogously by reversing the arguments. These notions lead to a gain--phase decomposition of the logarithmic norm.

\begin{proposition}[Gain--phase decomposition]
\label{prop:lognorm_phase}
Let $A$ be a bounded linear operator on $X$. Then
\begin{equation}
\label{eq:lognorm_phase}
  \mu(A) = \sup_{x\in S_X} \|Ax\|\cos_L(x,Ax).
\end{equation}
In particular, $\mu(A)\leq 0$ if and only if $\angle_L(x,Ax)\geq\pi/2$ for all 
$x\in S_X$ with $Ax\neq 0$.
\end{proposition}

In classical control theory, passivity is equivalent to a nonnegative
phase constraint on the transfer function~\cite{desoer1975feedback}. In normed spaces, directional phase plays an analogous role, characterizing key operator-theoretic properties, such as monotonicity, the incremental counterpart of passivity~\cite{desoer1975feedback}.

We call an operator $T$ \emph{monotone} if
\begin{equation} \label{eq:monotonicity}
\quad {\wpair{y}{x}\geq 0} \quad \forall \, (x,y)\in\igraph(T).    
\end{equation}
Our notion of monotonicity is weaker than the one in~\cite{davydov2024monotone}, which 
requires ${-\wpair{-y}{x}\geq 0}$. By~\cite[Theorem 3.5, (iii)]{proskurnikov2024regular}, the latter implies the former, but not conversely. To see this, take 
${x=(1,-1)}$ and ${y=(1,1)}$ in $\ell^\infty(\R^2)$ with the max 
pairing. Then ${\wpair{y}{x}_\infty=1\geq 0}$, yet 
${-\wpair{-y}{x}_\infty=-1<0}$. 

Monotonicity admits a simple characterization in terms of directional
phase.

\begin{lemma}[Phase characterization of monotonicity]
\label{lem:phase_monotone}
An operator $T$ is monotone if and only if 
${\angle_L(x,y)\leq\pi/2}$ for all ${(x,y)\in\igraph(T)}$ with 
${x,y\neq 0}$.
\end{lemma}

\noindent
Lemma~\ref{lem:phase_monotone} is the incremental counterpart of
the classical phase characterization of
passivity~\cite{desoer1975feedback}. For a bounded linear
operator~$A$, monotonicity is equivalent to
${\mu(-A) \leq 0}$. The formula
${\mu_1(A) = \mu_\infty(A^\top)}$ indeed recovers that
$A$ is $\ell^1$-monotone if and only if $A^\top$ is
$\ell^\infty$-monotone~\cite{bullo2024}, as illustrated numerically in
Section~\ref{sec:numerical}. In the next section, we strengthen
this connection by combining incremental phase and gain information into a
single object: the \gls{SRG} of an operator in a normed space.

\section{Scaled Relative Graphs}
\label{sec:srg}

In a Hilbert space, the \gls{SRG} is defined as in~\eqref{eq:srg_hilbert}
using the inner product to measure angles and
norms~\cite{ryu2022scaled}. In a normed space, the norm need not
arise from an inner product, but a compatible regular pairing
provides the necessary structure. Since the pairing need not be
symmetric, it induces two directional angles, and hence two
\glspl{SRG}; we define only the left variant, the right counterpart being
analogously defined.

\begin{definition}[Left SRG]
\label{def:srg_banach}
The \emph{left SRG} of an operator~$T$ is\footnote{Adopting the convention of~\cite{ryu2022scaled}, ${\angle_L(x,y)=0}$ whenever ${x=0}$ or ${y=0}$, so a pair with ${x\neq 0}$ and ${y=0}$ contributes the point $0$, and a pair with ${x=0}$ and ${y\neq 0}$ contributes the point $\infty$. By convention, the pair $(0,0)\in\igraph(T)$ does not contribute a point to $\SRGL(T)$, as $0/0$ is excluded by the arithmetic rules on~$\bar\C$~\cite{ryu2022scaled}.
}
\begin{equation}
\label{eq:srgl}
  \SRGL(T) = \left\{
    \tfrac{\|y\|}{\|x\|}e^{\pm i\angle_L(x,y)}
    \, \middle| \,
    (x,y)\in\igraph(T)
  \right\}  .
\end{equation}
\end{definition}

\noindent
The left \gls{SRG} reflects a range of properties of an operator. The modulus and argument of each point $z \in \SRG_L(T)$ coincide with the incremental gain and left phase of some input-output pair, respectively. By definition, conjugate symmetry about the real axis holds.  
In a Hilbert space, ${\SRG_L(T)=\SRG_R(T)}$ since directional angles coincide, so we simply write $\SRG(T)$. Otherwise, $\SRG_L(T) $ and $\SRG_R(T)$ are referred to as \emph{directional \glspl{SRG}}. Crucially, approximating directional \glspl{SRG} does not require an explicit representation of the operator~$T$, only a sufficient number of 
incremental input--output samples. For a bounded \textit{linear} operator $A$,
Proposition~\ref{prop:lognorm_phase} gives
$$\lognorm(A) = \sup_{z \in \SRG_L(A)} \Real(z).$$

\subsection{Operator properties from directional \glsentrytext{SRG}s}
\label{ssec:srg_props}

Directional \glspl{SRG} characterize geometrically key operator
properties~\cite{DavydovJafarpourBullo2022,davydov2024monotone,proskurnikov2024regular}.
An operator $T$ is 
\begin{itemize}[leftmargin=*, label=\raisebox{0.3ex}{\footnotesize$\bullet$}]
  \item \emph{$\ell$-Lipschitz}, with ${\ell>0}$, if
        ${\norm{y}\le\ell\norm{x}}$ for all ${(x,y)\in\igraph(T)}$;
        in particular, $T$ is \emph{nonexpansive} if ${\ell\le 1}$ and
        \emph{contractive} if ${\ell<1}$.

  \item \emph{one-sided $c$-Lipschitz}, with  ${c\in\R}$, if
        ${\wpair{y}{x}\le c\,\norm{x}^2}$ for all ${(x,y)\in\igraph(T)}$;
        in particular, $T$ is \emph{dissipative} if ${c=0}$.

  \item \emph{$\mu$-strongly monotone}, with ${\mu\ge 0}$, if
        ${\wpair{y}{x}\ge\mu\norm{x}^2}$ for all $(x,y)\in\igraph(T)$; in particular, $T$ is \emph{monotone} if ${\mu=0}$;

  \item \emph{$\gamma$-cocoercive}, with ${\gamma>0}$, if
        ${\wpair{y}{x}\ge\gamma\norm{y}^2}$ for all ${(x,y)\in\igraph(T)}$.
\end{itemize}
\noindent
In a Hilbert space, these properties recover the standard notions
of Lipschitz continuity, monotonicity, strong monotonicity, and
cocoercivity~\cite{bauschke2017}, and each has a natural
system-theoretic interpretation~\cite{chaffey2023graphical}. When the regular pairing is induced by the duality map of a Banach space,  one-sided $0$-Lipschitz and $\mu$-strong monotonicity specialize to 
accretivity and strong accretivity~\cite{Chidume2009}, respectively. In the following containment statements, half-planes in $\bar\C$ are understood to contain the point~$\infty$. 

\begin{theorem}[Operator properties from directional \glspl{SRG}]
\label{thm:srg_props}
Let $T$ be an operator on $X$. Then
\begin{enumerate}[label=\textup{(\alph*)}, noitemsep, leftmargin=*]
  \item\label{T:Lipschitz} 
    $T$ is $\ell$-Lipschitz if and only if  
    $${\SRGL(T)\subseteq\left\{z \in \bar{\mathbb{C}} \, \middle |\, |z|\leq\ell\right\}};$$
  \item\label{T:one-sided-Lipschitz} 
  $T$ is one-sided $c$-Lipschitz if and only if 
    $$\SRGL(T)\subseteq \left\{z \in \bar{\mathbb{C}} \, \middle |\, \Real z\leq c\right\};$$
  \item \label{T:strongly-monotone}  
  $T$ is $\mu$-strongly monotone if and only if 
    $$\SRGL(T)\subseteq \left\{z \in \bar{\mathbb{C}} \, \middle |\, \Real z\geq\mu\right\};$$
  \item \label{T:cocoercive}  
  $T$ is $\gamma$-cocoercive if and only if 
    $$\SRGL(T)\subseteq\left\{z \in \bar{\mathbb{C}} \ \middle | \ \left|z-\tfrac{1}{2\gamma}\right|\leq\tfrac{1}{2\gamma}\right\}.$$
\end{enumerate}
\end{theorem}

\noindent
Theorem~\ref{thm:srg_props} extends to normed spaces the graphical
framework introduced in~\cite{ryu2022scaled} and further developed
in~\cite{chaffey2023graphical}. It provides geometric containment
tests for verifying properties of~$T$ directly from $\SRGL(T)$. For
example, $T$ is monotone if and only if $\SRGL(T)$ lies in the
closed right half-plane. 
Another immediate corollary is a graphical contractivity test.
Suppose that $X$ is complete and that $T\colon X\to X$. If
\begin{equation}
\label{eq:contraction_factor}
\ell
\coloneqq
\sup_{z\in\SRGL(T)}|z|
<1,
\end{equation}
then Theorem~\ref{thm:srg_props} implies that $T$ is
$\ell$-contractive. By the Banach contraction
principle~\cite{Zeidler1986},  it follows that $T$ has a unique fixed point and the
iteration ${x^{k+1}=T(x^k)}$ converges to it geometrically with
rate~$\ell$.

\subsection{Directional \glsentrytext{SRG} calculus}

Motivated by the problem of inferring the directional \gls{SRG} 
of a composite operator from those of its components, we extend the \gls{SRG} calculus
of~\cite{ryu2022scaled,chaffey2023graphical}. Throughout this section, we focus on real normed spaces equipped with \glspl{SIP}, as the algebraic rules rely on linearity in the first argument.

For ${S\subseteq\bar\C}$ and ${\alpha\in\R}$, we write
${\alpha S=\{\alpha z \mid z\in S\}}$ and ${S^{-1}=\{1/\bar z \mid z\in S\}}$, where ${1/0=\infty}$, ${1/\infty=0}$, and ${0S=\{0\}}$ for any nonempty ${S\subseteq\bar\C}$.
For $S_1, S_2 \subseteq \bar{\C}$, define
\begin{equation}\nonumber
  S_1 \boxplus S_2
  \!=\! \left\{ z \,\middle|\, \exists\, z_j \in S_j,\!\!
    \begin{array}{l}
      \Real z = \Real z_1 + \Real z_2, \\
      \bigl||z_1|-|z_2|\bigr| \leq |z| \leq |z_1|+|z_2|
    \end{array} \!\!\!
  \right\}\!.
\end{equation}
Given an operator $A$ and sets $S_1, S_2 \subseteq \bar{\C}$, we also
define
\begin{equation}\nonumber
  S_1  \diamond_A   S_2
  \!=\! \left\{ z \,\middle|\, \exists\, z_j \in S_j,\!\!
    \begin{array}{l}
      |z| = |z_1|\,|z_2|, \\
      |\Real z - \Real z_1\,\Real z_2| \leq \sigma_A\,|z|
    \end{array} \!\!\!
  \right\}\!,
\end{equation} 
\noindent
where
\[
\sigma_A
=
\sup_{\substack{(x,y)\in\igraph(A)\\ x,y\neq0}}
\left\|
\frac{y}{\|y\|}
-
\cos_L(x,y)\frac{x}{\|x\|}
\right\|.
\]
We set ${\sigma_A =0}$ when the supremum is taken over the empty set.
For ${z_1,z_2\in\bar\C}$, we write ${z_1\boxplus z_2}$ and
${z_1\diamond_A z_2}$ for ${\{z_1\}\boxplus\{z_2\}}$ and
${\{z_1\}\diamond_A\{z_2\}}$, respectively. The constraints on real parts apply
only when all points involved are finite. Accordingly, we further set
${z_1\boxplus z_2=\{\infty\}}$ when ${z_1=\infty}$ or
${z_2=\infty}$, and ${z_1\diamond_A z_2=\bar\C}$ when
${\{|z_1|,|z_2|\}=\{0,\infty\}}$.

\begin{theorem}[\gls{SRG} calculus]
\label{thm:srg_ops}
Let $A,B$ be operators on a real normed space $(X,\|\cdot\|)$ with a
compatible \gls{SIP} $\wpair{\cdot}{\cdot}$. Then
\begin{enumerate}[label=\emph{(\alph*)},noitemsep,leftmargin=*]
  \item\label{srg:scaling}
    \emph{Scaling:} ${\SRGL(\alpha A)=\alpha\,\SRGL(A)}$ for every
    ${\alpha\in\R}$;
  \item\label{srg:inversion}
    \emph{Inversion:}
    ${\SRGL(A^{-1})=\bigl(\SRGR(A)\bigr)^{-1}}$;
    \item\label{srg:addition}
    \emph{Addition:}
    ${\SRGL(A+B)\subseteq\SRGL(A)\boxplus\SRGL(B)}$;
    \item\label{srg:composition}
    \emph{Composition:} ${\SRGL(AB)\subseteq
      \SRGL(A)\diamond_A\SRGL(B)}$. 
\end{enumerate}
\end{theorem}

\noindent
Theorem~\ref{thm:srg_ops} extends to normed spaces the \gls{SRG} calculus of~\cite{ryu2022scaled,chaffey2023graphical}, enabling properties of composite operators to be inferred from those of their components. For example, one directly recovers submultiplicativity of Lipschitz constants and additivity of one-sided Lipschitz
constants~\cite{davydov2024monotone}. However, two key differences distinguish our result from its Hilbert space counterpart. First, inversion exchanges left and right \glspl{SRG}, indicating that both directional \glspl{SRG} are essential for analysis. Second, composition and addition yield new \gls{SRG}  operations ($\diamond_A$ and~$\boxplus$), rather than the Minkowski product and sum available in Hilbert spaces~\cite{ryu2022scaled,chaffey2023graphical}.  

\section{Numerical Case Studies}
\label{sec:numerical}

We now illustrate how directional \glspl{SRG} differ from
their Hilbert space counterparts, and how these differences can be
exploited in applications, such as dynamic programming. 
The first two case studies illustrate how monotonicity may hold in $\ell^1$ or $\ell^\infty$ while failing in $\ell^2$. The third compares a graphical contractivity estimate obtained by sampling a directional \gls{SRG} with an analytical Lipschitz bound. All directional \glspl{SRG} are constructed from randomly sampled incremental input--output pairs. The unifying message is that operator properties can be verified graphically by checking if the \glspl{SRG} lie within given regions.

\subsection{Monotonicity of linear operators}
Consider the matrices
\begin{equation}
\label{eq:matrices}
  A_1 = \left[\begin{array}{rrr} 0 & -2 & -2 \\ 0 & 2 & -1 \\ 0 & 0 & 3 \end{array}\right]\! ,
  \ \,
  A_\infty = \left[\begin{array}{rrr} 0 & 0 & \,0 \\ -2 & 2 & \,0 \\ -2 & -1 & \,3 \end{array}\right] \! .
\end{equation}
Note that ${A_\infty = A_1^\top}$. Viewed as linear operators, neither is $\ell^2$-monotone, as their symmetric part has a negative eigenvalue. By contrast, ${\mu_1(-A_1)=0}$ and ${\mu_\infty(-A_\infty)=0}$, so $A_1$ is $\ell^1$-monotone and $A_\infty$ is $\ell^\infty$-monotone. This is not a coincidence: the formula
${\mu_1(A) = \mu_\infty(A^\top)}$ ensures that
$\ell^1$-monotonicity of $A_1$ is equivalent to
$\ell^\infty$-monotonicity of
$A_1^\top = A_\infty$~(cf.\ Remark~\ref{rem:duality}).

Fig.~\ref{fig:srg_linear} shows the \glspl{SRG} of $A_1$ (top) and $A_\infty$ (bottom) in $\ell^1$ (left), $\ell^2$ (center), and $\ell^\infty$ (right), with the max pairing in $\ell^\infty$. For $A_1$, the $\ell^1$ \gls{SRG} (top-left) lies in the right half-plane, confirming $\ell^1$-monotonicity. By contrast, the $\ell^2$ and $\ell^\infty$ \glspl{SRG} extend into the left half-plane, so monotonicity does not hold in these norms. For $A_\infty$, a dual picture emerges: the $\ell^\infty$ \gls{SRG} (bottom-right) lies in the right half-plane, while the $\ell^1$ and $\ell^2$ \glspl{SRG} extend into the left half-plane. 

\begin{figure}[t!]
  \centering
  \includegraphics[width=\columnwidth]{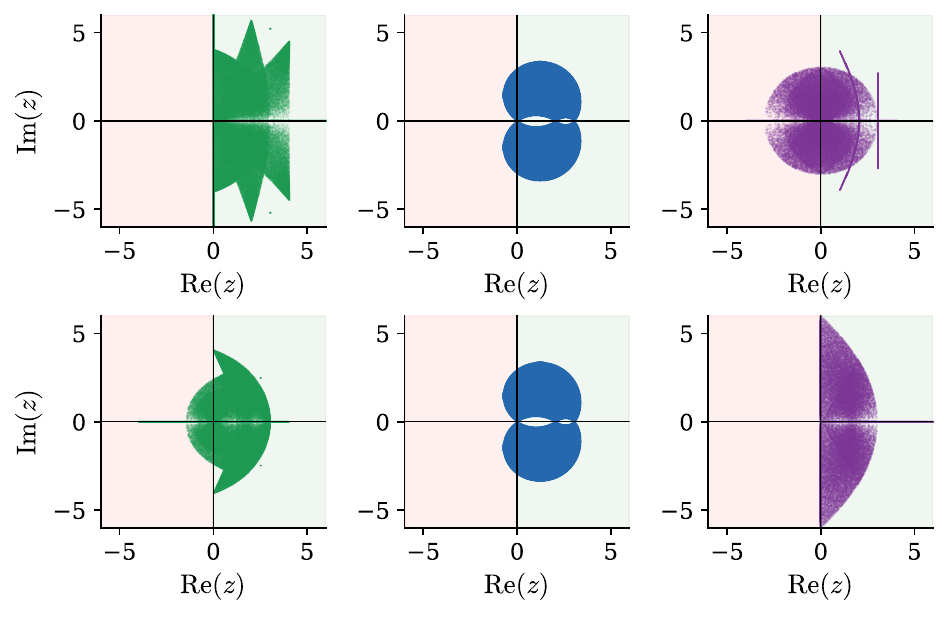}
  \caption{Left \glspl{SRG} of $A_1$ (top) and 
  $A_\infty$ (bottom) in $\ell^1$ (left), $\ell^2$ (center), 
  and $\ell^\infty$ (right; max pairing). The left half-plane marks 
  monotonicity violations: $A_1$ is $\ell^1$-monotone and $A_\infty$ is $\ell^\infty$-monotone. Neither is $\ell^2$-monotone.}
  \label{fig:srg_linear}
\end{figure}

\subsection{Monotonicity of nonlinear operators}
Consider the matrices~\eqref{eq:matrices} and, for ${p \in \{1,\infty\},}$ define
\begin{equation}
\label{eq:nl_map}
    F_p(x) = \diag(A_p)\varphi(x) + (A_p - \diag(A_p))x,
\end{equation}
where $\diag(A)$ is the diagonal matrix with the same diagonal as~$A$ and 
$\varphi$ acts componentwise as
${\varphi_i(x_i) = x_i + x_i^3}$. The map $\varphi$ is
sign-preserving (${\sign(\varphi_i(x_i)) = \sign(x_i)}$) and
expansive ($|\varphi_i(x_i) - \varphi_i(y_i)| \geq |x_i - y_i|$),
so it ``reinforces'' the diagonal dominance underlying
$\ell^1$- and $\ell^\infty$-monotonicity of the linear parts. By
contrast, $\ell^2$-monotonicity fails along the negative eigenspace
of the symmetric part.

Fig.~\ref{fig:srg_nonlinear} shows the \glspl{SRG} of 
$F_1$ (top) and $F_\infty$ (bottom). As in the linear case, the $\ell^1$ \gls{SRG} of 
$F_1$ (top-left) and the $\ell^\infty$ \gls{SRG} of 
$F_\infty$ (bottom-right) lie in the right half-plane, while 
all $\ell^2$ \glspl{SRG} extend into the left half-plane. 
Compared with Fig.~\ref{fig:srg_linear}, the \glspl{SRG} occupy a broader region of the extended complex plane, but the monotonicity verdicts are identical.

\subsection{Sampled contraction estimates for Bellman operators}

Directional \glspl{SRG} in $\ell^\infty$ find a natural use in
dynamic programming~\cite{sutton2018reinforcement,bertsekas2012dynamic}. Consider a \gls{MDP} with finite state space 
${\mathsf{S}=\{1,\ldots,n\}}$, finite action space ${\mathsf{A} =\{1,\ldots,m\}}$,  
reward function ${r\colon \mathsf{S}\times \mathsf{A}\to\R}$, discount 
factor $\gamma\in(0,1)$, and transition probabilities 
$${P(s'\!\mid\! s,a)=\Pr(s_{t+1}\!=\!s'\!\mid\! s_t\!=\!s,a_t\!=\!a) ,\ s,s'\in \mathsf{S}, a\in \mathsf{A}.}$$   
Every stationary policy $\pi\colon \mathsf{S}\to \mathsf{A}$ induces a reward vector $r_\pi\in\R^n$, with $(r_\pi)_s=r(s,\pi(s))$, and a row-stochastic 
transition matrix $P_\pi\in\R^{n\times n}$, with 
${(P_\pi)_{ss'}=P(s'\mid s,\pi(s))}$.

\paragraph{Standard policy evaluation}
The \emph{policy evaluation operator} 
$T_\pi\colon\R^n\to\R^n$ is defined by~\cite[\S 4.1]{sutton2018reinforcement},\cite[\S 1.1]{bertsekas2012dynamic}
\begin{equation}
\label{eq:bellman}
  T_\pi v = r_\pi + \gamma P_\pi v.
\end{equation}
Its unique fixed point ${v_\pi=(I-\gamma P_\pi)^{-1}r_\pi}$ defines the 
value function associated with~$\pi$~\cite{sutton2018reinforcement,bertsekas2012dynamic}. Value iteration 
${v^{k+1}=T_\pi v^k}$ converges geometrically to $v_\pi$ because $T_\pi$ is 
$\gamma$-contractive in $\ell^\infty$~\cite{bertsekas2012dynamic}. 
By Theorem~\ref{thm:srg_props}, this is equivalent to $\SRGL(T_\pi)$
being contained in the disk $\{z \in \bar{\mathbb{C}} \,|\,  |z|\leq\gamma\}$.
Fig.~\ref{fig:srg_bellman} (left) illustrates this for a 
randomly generated $8$-state MDP with $\gamma=0.7$. The sampled left \gls{SRG} points in $\ell^\infty$ (with max pairing) lie in the disk of radius $\gamma$ centered at the origin.

\begin{figure}[t!]
  \centering
  \includegraphics[width=\columnwidth]{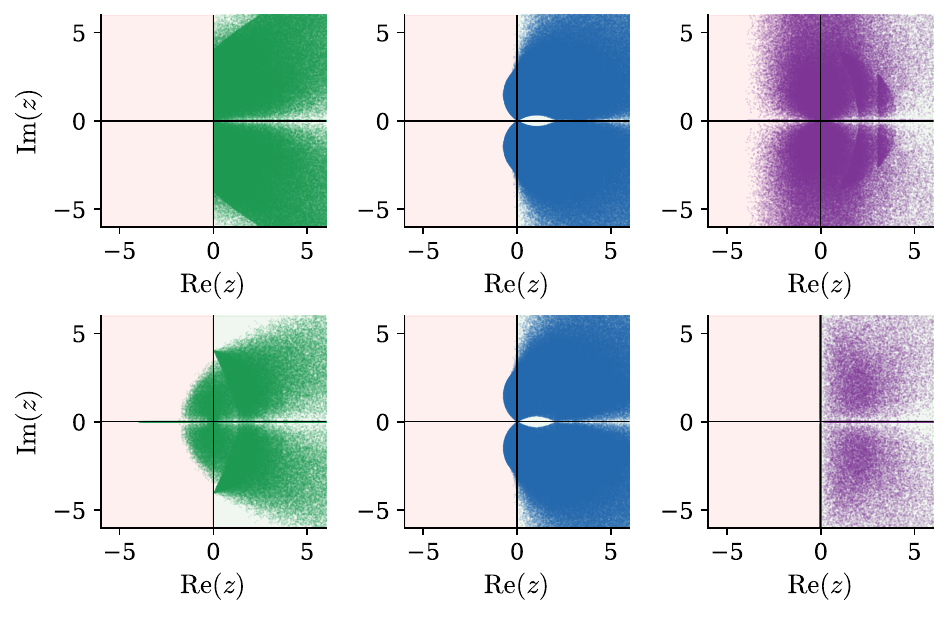}
  \caption{Left \glspl{SRG} of $F_{1}$ (top) and $F_{\infty}$ (bottom) in $\ell^1$ 
  (left), $\ell^2$ (center), and $\ell^\infty$ 
  (right; max pairing). The left half-plane marks 
  monotonicity violations: $F_{1}$ is $\ell^1$-monotone and $F_{\infty}$ is $\ell^\infty$-monotone. Neither is $\ell^2$-monotone.}
  \label{fig:srg_nonlinear}
\end{figure}

\paragraph{Regularized policy evaluation}
In approximate dynamic programming, a common strategy to incorporate prior  knowledge is to regularize the Bellman operator~\cite{bertsekas2012dynamic,GeistScherrerPietquin2019}.
Consider the 
\emph{regularized policy evaluation operator}
\begin{equation}
\label{eq:bellman_reg}
  {T}_{\pi,\alpha} v = T_\pi v + \alpha \varphi(v),
\end{equation}
where $\alpha>0$ is a regularization parameter and 
$\varphi$ acts componentwise as 
$\varphi_i(v_i)=-v_i/(1+|v_i|)$. 
The map $\varphi$ is bounded (${|\varphi_i(v_i)|\leq 1}$), sign-reversing 
(${\sign(\varphi(v))=-\sign(v)}$), and dissipative: 
it penalizes large values by shrinking them toward zero.  
The map $\varphi$ is also $1$-Lipschitz,  so the triangle inequality gives
\begin{equation}
\label{eq:analytical_bound}
  \|{T}_{\pi,\alpha} v - {T}_{\pi,\alpha} w\|_\infty
  \leq (\gamma + \alpha)\|v-w\|_\infty.
\end{equation}
For $\gamma+\alpha<1$, this certifies contraction, but the bound may 
be conservative, as it ignores the interaction between $T_\pi$ and $\varphi$, effectively treating them as adversarially aligned.

Fig.~\ref{fig:srg_bellman} (right) shows 
$\SRGL({T}_{\pi,\alpha})$ in $\ell^\infty$ for $\alpha=0.25$. The Lipschitz 
bound~\eqref{eq:analytical_bound} gives $\gamma+\alpha=0.95$ (dotted circle), yet all sampled points in $\SRGL({T}_{\pi,\alpha})$ lie in a disk of radius~${\ell \approx 0.89}$ (dashed circle). This numerical comparison suggests that the analytical
bound may be conservative, although the sampled maximum is not an
upper bound on the full directional \gls{SRG} and therefore does
not by itself constitute a contraction certificate.
A key advantage of graphical analysis is that approximate certificates can be obtained without explicit knowledge of ${T}_{\pi,\alpha}$. Directional \glspl{SRG} can be constructed directly from sampled input--output pairs, even for operators accessible only through black-box evaluations or simulation, as is often the case in practice with Bellman operators~\cite{sutton2018reinforcement}.

\section{Conclusion}
\label{sec:conclusion}

We have extended the \gls{SRG} framework from Hilbert spaces to
normed spaces by replacing the inner product with a regular
pairing~\cite{proskurnikov2024regular}, whose asymmetry gives rise to directional angles and, in turn, directional \glspl{SRG}. The resulting characterizations of operator properties in terms of geometric containment tests (Theorem~\ref{thm:srg_props}) and \gls{SRG} calculus rules (Theorem~\ref{thm:srg_ops}) provide a coherent graphical framework for systems with bounded or sparse signals, as illustrated numerically on
static maps and Bellman operators.

Several directions remain open for future work. Natural extensions include a systematic duality theory and a feedback
stability theorem in the spirit of~\cite{chaffey2023graphical}
connecting our framework with classical control
theory~\cite{desoer1975feedback}.  On the computational side, sample
complexity bounds and efficient approximation algorithms for
constructing \glspl{SRG} from data are promising next steps. Finally, validating directional \glspl{SRG} in applications, including  $\ell^1$ optimal control~\cite{dahleh1994control}, dynamic programming~\cite{bertsekas2012dynamic,sutton2018reinforcement}, and the analysis of fixed-point iterations~\cite{Zeidler1986,Chidume2009}, is equally important.

\begin{figure}[t]
  \centering
  \includegraphics[width=\columnwidth]{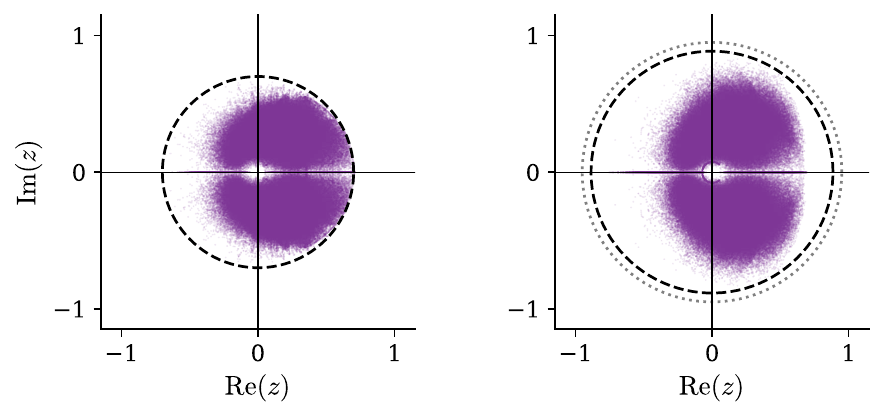}
  \caption{Left \glspl{SRG} (max pairing) in $\ell^\infty$  of 
 policy evaluation operator $T_\pi$ and regularized policy evaluation operator ${T}_{\pi,\alpha}$, with ${\alpha=0.25}$, for a randomly generated $8$-state \gls{MDP} with ${\gamma=0.7}$.
  \emph{Left:} $\SRGL({T}_\pi)$ lies in a disk of radius~${\gamma}$, confirming ${T}_\pi$ is $\gamma$-contractive.
  \emph{Right:} $\SRGL({T}_{\pi,\alpha})$ lies in a disk of radius~${\ell \approx 0.89}$, thus improving on the Lipschitz bound ${\alpha + \gamma = 0.95}$.}
  \label{fig:srg_bellman}
\end{figure}

\bibliographystyle{IEEEtran}
\bibliography{refs}

\appendix

\section{Proofs}
\label{sec:proofs}

\begin{proof}[Proof of Lemma~\ref{lem:cosine_properties}]
\ref{cos:self} Fix ${x\in X\!\setminus\!\{0\}}$.  By compatibility of the regular pairing with the norm,
${\wpair{x}{x}=\|x\|^2}$, so
${\cos_L(x,x) 
=1.}$ 
The straight angle property~\ref{rwp:straight} gives
$\wpair{-x}{x}=-\|x\|^2,$ hence
${\cos_L(x,-x) 
=-1 .}$

\ref{cos:scale} Fix ${x,y\in X\!\setminus\!\{0\}}$ and ${ \sigma >0}$. 
By weak homogeneity, 
${\wpair{\sigma x}{y} = \wpair{x}{\sigma y} = \sigma\wpair{x}{y}}.$ 
By homogeneity of the norm, one also has ${\|\sigma x\|=\sigma\|x\|}$. Then 
${\cos_L(x,y) = \cos_L(\sigma x,y) = \cos_L(x,\sigma y)}.$

\ref{cos:convex} Fix nonzero ${x,y,z\in X}$ and ${\alpha,\beta\geq 0}$, with ${\alpha y+\beta z\neq 0}$. Subadditivity and weak homogeneity of the regular pairing 
give 
${\wpair{\alpha y+\beta z}{x} \le \alpha\wpair{y}{x}+\beta\wpair{z}{x}.}$ Dividing by $\norm{\alpha y+\beta z}\norm{x}$ yields the claim.

\ref{cos:defect} Let ${x,y,z\in S_X}$ and ${q=z-\wpair{z}{y}\,y}$.  By linearity in the first argument of $\wpair{\cdot}{\cdot}$, one obtains
${\wpair{q}{y}=0}$ and
$${\wpair{z}{x}=\wpair{z}{y}\,\wpair{y}{x}+\wpair{q}{x}}.$$
Since ${x,y,z\in S_X}$, the second identity can be expressed as
$${\cos_L(x,z)-\cos_L(y,z)\,\cos_L(x,y)=\wpair{q}{x}},$$ 
which, by the Cauchy--Schwarz inequality, yields the claim.
\end{proof}

\begin{proof}[Proof of Proposition~\ref{prop:lognorm_phase}]
For ${x\in S_X}$, one has
\[
  \wpair{Ax}{x}
  = \|Ax\|\|x\|\cos_L(x,Ax)
  = \|Ax\|\cos_L(x,Ax).
\]
By Lumer's identity~\eqref{eq:lumer}, taking the supremum over ${x\in S_X}$ yields~\eqref{eq:lognorm_phase}. Since $\|Ax\|\geq 0$, the sign of $\|Ax\|\cos_L(x,Ax)$ is determined by $\cos_L(x,Ax)$ whenever $Ax\neq 0$.  Hence, $\mu(A)\leq 0$ if and only if $\cos_L(x,Ax)\leq 0$ for all
$x\in S_X$ with $Ax\neq 0$, i.e.,
$\angle_L(x,Ax)\geq\pi/2$.
\end{proof}

\begin{proof}[Proof of Lemma~\ref{lem:phase_monotone}]
Fix ${(x,y)\in\igraph(T)}$, with ${x,y\neq 0}$. Monotonicity requires
${\wpair{y}{x}\geq 0}$.  Since ${\norm{x}>0}$, ${\norm{y}>0}$, and
${\wpair{y}{x}=\|x\|\|y\|\cos_L(x,y)},$ 
the sign is determined by ${\cos_L(x,y)\geq 0}$, which holds if and only
if  ${\angle_L(x,y)\leq\pi/2.}$
\end{proof}

\begin{proof}[Proof of Theorem~\ref{thm:srg_props}]
Fix $(x,y)\in\igraph(T)$ with $x\neq 0$, and let ${z\in\SRGL(T)}$ be
the corresponding SRG point. Then ${|z|=\|y\|/\|x\|}$ and
\begin{equation}
\label{eq:rez_pairing}
\Real z = \frac{\wpair{y}{x}}{~\|x\|^2}.
\end{equation}
A pair with ${x=0}$ and ${y\neq 0}$ contributes the point
$\infty$ and satisfies the defining inequalities
in~\ref{T:one-sided-Lipschitz} and~\ref{T:strongly-monotone}, but
neither of those in~\ref{T:Lipschitz} and~\ref{T:cocoercive},
consistently with our conventions. 

\ref{T:Lipschitz}
By definition, ${|z|=\|y\|/\|x\|}$, so $T$ is $\ell$-Lipschitz if and only if $|z|\leq\ell$ for all
$z\in\SRGL(T)$.  The nonexpansive and contractive cases correspond
to $\ell\leq 1$ and $\ell<1$, respectively.

\ref{T:one-sided-Lipschitz}
From~\eqref{eq:rez_pairing}, $T$ is one-sided $c$-Lipschitz
if and only if $\Real z\leq c$ for every $z\in\SRGL(T)$. The dissipative case corresponds to $c=0$.

\ref{T:strongly-monotone}
From~\eqref{eq:rez_pairing}, $T$ is $\mu$-strongly monotone
if and only if $\Real z\geq\mu$ for
every $z\in\SRGL(T)$.  The monotone case corresponds to $\mu=0$.

\ref{T:cocoercive}
The operator $T$ is $\gamma$-cocoercive if 
$\wpair{y}{x}\geq\gamma\|y\|^2$ for all $(x,y)\in\igraph(T)$.
Dividing by $\|x\|^2$, using~\eqref{eq:rez_pairing}, and recalling
$|z|=\|y\|/\|x\|$ gives $\Real z\geq\gamma|z|^2$.  Writing
${z=a+\mathrm{i}b,}$ the inequality $a\geq\gamma(a^2+b^2)$ rearranges to
$$\left(a-\tfrac{1}{2\gamma}\right)^2+b^2\leq\left(\tfrac{1}{2\gamma}\right)^2,$$ which is
the claimed disk containment.
\end{proof}

\begin{proof}[Proof of Theorem~\ref{thm:srg_ops}]

\emph{\ref{srg:scaling}~Scaling.}
For ${\alpha=0}$, 
$${\SRGL(0)=\{0\} = 0\cdot\SRGL(A)}.$$
For ${\alpha\neq 0}$, fix ${(x,y)\in\igraph(A)}$, with ${x\neq 0}$,
and let ${z\in\SRGL(A)}$ be the corresponding \gls{SRG} point, so that
${|z|=\tfrac{\|y\|}{\|x\|}}$ and ${\Real z=\tfrac{\wpair{y}{x}}{\|x\|^2}}.$
Then ${(x,\alpha y)\in\igraph(\alpha A)}$ implies
${z'\in\SRGL(\alpha A)}$, with
${|z'|=\tfrac{\|\alpha y\|}{\|x\|}=|\alpha|\,|z|}.$
By linearity in the first argument of $\wpair{\cdot}{\cdot}$, one obtains
$${\cos_L(x,\alpha y)=\sign(\alpha)\cos_L(x,y)},$$
so conjugate symmetry gives ${z'=\alpha z}$. Pairs with ${x=0}$ contribute $\infty$ to both \glspl{SRG}, and ${\alpha\,\infty=\infty}$ for ${\alpha\neq 0}$. 

\emph{\ref{srg:inversion}~Inversion.}
Let ${(x,y)\in\igraph(A)}$ with ${x,y\neq 0}$.
Then ${(y,x)\in\igraph(A^{-1})}$ yields
${z'\in\SRGL(A^{-1})}$ with
\[
  |z'|=\frac{\|x\|}{\|y\|},
  \qquad
  \angle z'=\angle_R(x,y).
\]
Similarly, ${(x,y)\in\igraph(A)}$ yields ${z\in\SRGR(A)}$ with
\[
  |z|=\frac{\|y\|}{\|x\|},
  \qquad
  \angle z=\angle_L(y,x).
\]
Since ${\angle_R(x,y)=\angle_L(y,x)}$, one obtains ${z'=1/\bar{z}}$.
Moreover, ${\infty\in\SRGR(A)}$ if and only if $A$ is
multi-valued, \textit{i.e.}, ${0\in\SRGL(A^{-1})}$.
Similarly, ${0\in\SRGR(A)}$ if and only if
${\infty\in\SRGL(A^{-1})}$.
The claim follows by conjugate symmetry.

\emph{\ref{srg:addition}~Addition.}
Let ${(x,y_A)\in\igraph(A)}$, ${(x,y_B)\in\igraph(B)}$,
so that ${(x,y_A+y_B)\in\igraph(A+B)}$, with
corresponding points ${z_A,z_B,z}$.
If ${x=0}$ and ${y_A+y_B=0}$, the pair ${(0,0)}$ contributes no point
to ${\SRGL(A+B)}$. If ${x=0}$ and ${y_A+y_B\neq 0}$, then 
${z=\infty}$ and ${\infty\in\SRGL(A)}$ (if ${y_A\neq 0}$) or
${\infty\in\SRGL(B)}$ (if ${y_A=0}$), so 
$${z=\infty\in\SRGL(A)\boxplus\SRGL(B)}.$$ 
If ${x\neq 0}$, linearity of $\wpair{\cdot}{\cdot}$ in its
first argument gives
\begin{align*}
    \Real z
  &=\frac{\wpair{y_A+y_B}{x}}{\norm{x}^2} 
  =\frac{\wpair{y_A}{x}+\wpair{y_B}{x}}{\norm{x}^2} \\
  &=\Real z_A+\Real z_B.
\end{align*}
The triangle and reverse-triangle inequalities give
\[
  \bigl|\norm{y_A}-\norm{y_B}\bigr|
  \leq\norm{y_A+y_B}
  \leq\norm{y_A}+\norm{y_B}.
\]
Dividing by ${\norm{x}}$ yields
${\bigl||z_A|-|z_B|\bigr|\leq|z|\leq|z_A|+|z_B|}.$

\emph{\ref{srg:composition}~Composition.}
Let ${(x,y)\in\igraph(B)}$, ${(y,z)\in\igraph(A)}$,
so that ${(x,z)\in\igraph(AB)}$, with corresponding
points ${w_B,w_A,w}$. If ${x=z=0}$, the pair contributes no point to $\SRGL(AB)$.
If ${x=0}$ and ${z\neq 0}$, then ${w=\infty}$, and ${w_B=\infty}$
when ${y\neq 0}$, while ${w_A=\infty}$ if ${y=0}$. If ${x\neq 0}$
and ${z=0}$, then ${w=0}$, and ${w_A=0}$ if ${y\neq 0}$, while
${w_B=0}$ if ${y=0}$. If ${x,z\neq 0}$ and ${y=0}$, then
${w_B=0}$, ${w_A=\infty}$, and ${0\diamond_A\infty=\bar\C}$. In each
case ${w\in\SRGL(A)\diamond_A\SRGL(B)}$.
If ${x,y,z\neq 0}$, then
\[
  |w|=\frac{\norm{z}}{\norm{x}}
  =\frac{\norm{z}}{\norm{y}}\,\frac{\norm{y}}{\norm{x}}
  =|w_A|\,|w_B|.
\]
By Lemma~\ref{lem:cosine_properties},
with ${\hat{y}=y/\norm{y}}$ and ${\hat{z}=z/\norm{z}}$,
\begin{align*}
  \bigl|\cos_L(x,z)-\cos_L(y,z)\,\cos_L(x,y)\bigr|
  &\leq\norm{\hat{z}-\cos_L(y,z)\,\hat{y}}\\
  &\leq\sigma_A,
\end{align*}
where the second inequality follows from the definition
of~$\sigma_A$.
Since ${\Real w=|w|\cos_L(x,z)}$,
$\Real w_A=|w_A|\cos_L(y,z)$, and
$\Real w_B=|w_B|\cos_L(x,y)$, multiplying by
$|w|=|w_A|\,|w_B|$ gives
$${\bigl|\Real w-\Real w_A\,\Real w_B\bigr|
\leq\sigma_A\,|w|}.$$
\end{proof}

\end{document}